\documentclass[12pt]{amsart}
\usepackage{amscd,amssymb,amsthm,verbatim}
\newcommand{\cone}{\operatorname{cone}}
\newcommand{\const}{\operatorname{const.}}

\newcommand{\diam}{\operatorname{diam}}

\newcommand{\dvol}{\operatorname{dvol}}

\newcommand{\inj}{\operatorname{inj}}

\newcommand{\R}{{\mathbb R}}

\newcommand{\Ric}{\operatorname{Ric}}
\newcommand{\Riem}{\operatorname{Riem}}

\newcommand{\Rm}{\operatorname{Rm}}

\newcommand{\SO}{\operatorname{SO}}

\newcommand{\vol}{\operatorname{vol}}
\newcommand{\Z}{{\mathbb Z}}

\numberwithin{equation}{section}
\theoremstyle{plain}

\newtheorem{lemma}[equation]{Lemma}
\newtheorem{theorem}[equation]{Theorem}
\newtheorem{proposition}[equation]{Proposition}
\newtheorem{corollary}[equation]{Corollary}

\newtheorem{conjecture}[equation]{Conjecture}

\theoremstyle{remark}

\newtheorem{remark}[equation]{Remark}

\usepackage{graphicx}
\input{amssym.def}
\input{amssym}
\input{epsf}
\usepackage{a4wide}

\begin{document}

\title{On $3$-manifolds with pointwise pinched nonnegative Ricci curvature}

\author{John Lott}
\address{Department of Mathematics\\
University of California, Berkeley\\
Berkeley, CA  94720-3840\\
USA} \email{lott@berkeley.edu}

\thanks{Research partially supported by NSF grant
DMS-1810700}
\date{February 16, 2023}

\begin{abstract}
  There is a conjecture that a complete Riemannian $3$-manifold with bounded
  sectional curvature, and pointwise pinched nonnegative Ricci curvature, must be flat or compact.  We show that this is true when the negative part (if any) of the sectional curvature decays quadratically.
\end{abstract}

\maketitle


\section{Introduction} \label{sect1}

Let $(M, g)$ be a complete connected Riemannian $3$-manifold.
Suppose that $\Ric(M,g) \ge 0$. At a point $m \in M$, the
Ricci tensor on $T_mM$ can be diagonalized relative to $g(m)$.
Let $r_1 \le r_2 \le r_3$ be its eigenvalues. Given $c \in (0,1]$,
we say that $(M,g)$ is {\em $c$-Ricci pinched} if at all $m \in M$,
we have $r_1 \ge c r_3$. 

\begin{conjecture} \label{1.1}
  Let $(M, g)$ be a complete connected Riemannian manifold of dimension
  three,
  with bounded sectional curvature and nonnegative Ricci curvature.
  Suppose that $(M, g)$ is $c$-Ricci pinched for some $c \in (0,1]$.
  Then $(M,g)$ is flat or $M$ is compact.
\end{conjecture}

Using basic properties of Ricci flow, one can show that Conjecture \ref{1.1} is
equivalent to the following conjecture.

\begin{conjecture} \label{1.2}
  Let $(M, g)$ be a complete connected Riemannian manifold of dimension
  three,
  with bounded sectional curvature and positive Ricci curvature.
  Suppose that $(M, g)$ is $c$-Ricci pinched for some $c \in (0,1]$.
  Then $M$ is compact.
\end{conjecture}

We will think of Conjectures \ref{1.1} and \ref{1.2}
interchangeably.
They are apparently due to Hamilton,
who proved a result similar to Conjecture \ref{1.2}
for hypersurfaces in Euclidean space \cite{Hamilton (1994)}.
Conjecture \ref{1.2} can be considered to be a scale-invariant version of the
Bonnet-Myers theorem.  The latter says that if a complete Riemannian $n$-manifold
$(M, g)$ has $\Ric \ge (n-1) k^2 g$, with $k > 0$, then
$M$ is compact with diameter at most $\frac{\pi}{k}$. In Conjecture \ref{1.2},
rather than an explicit bound for the diameter, the claim is that
the diameter is finite.

To get a feeling why Conjecture \ref{1.1} might be true, consider a
Riemannian manifold $(M,g)$  with nonnegative Ricci curvature that is
strictly conical outside of a compact subset.  The
Ricci curvature vanishes in the radial direction of the cone.  The $c$-Ricci
pinching then
implies that $M$ is Ricci-flat on the conical region and hence flat there,
since the dimension is three. Then the link of the cone
consists of copies of round 
$S^2$'s and $\R P^2$'s. From the splitting theorem,
the link must be connected.
Since it bounds a compact $3$-manifold, it must be
$S^2$. The global nonnegativity of the
Ricci curvature now
implies that $M$ is isometric to $\R^3$.
This intuition will enter into the proof of
Theorem \ref{1.4} below.

We show that Conjectures \ref{1.1} and \ref{1.2} are true under an
extra curvature assumption.

\begin{theorem} \label{1.3}
  Conjecture \ref{1.1} is true if \\
  a. $(M,g)$ has nonnegative sectional curvature, or\\
  b. $(M,g)$ has quadratic curvature decay.
\end{theorem}

Theorem \ref{1.3}(a) was proven earlier in
\cite{Chen-Zhu (2000)} using Ricci flow. Our proof also uses Ricci flow but is technically different.

\begin{theorem} \label{1.4}
  Conjecture \ref{1.1} is true if there is some $A < \infty$ so that
      the sectional curvatures of $(M, g)$ satisfy
      $K(m) \ge - \: \frac{A}{d(m,m_0)^2}$, where $m_0$ is
      some basepoint. 
  \end{theorem}

Theorem \ref{1.4} implies Theorem \ref{1.3}
but we state it separately, since the
proof of Theorem \ref{1.4} uses results from the preprint
\cite{Lebedeva-Petrunin}.

Besides the particular results in Theorems \ref{1.3} and \ref{1.4},
we prove more general results that may lead to a
proof of Conjecture \ref{1.1}. The next proposition says that if $(M,g_0)$
is noncompact and
satisfies the hypotheses of Conjecture \ref{1.1} then the ensuing Ricci flow
exists for all positive time and is type-III.

\begin{proposition} \label{1.5}
  Given $(M, g_0)$ as in Conjecture \ref{1.1} with $M$ noncompact,
  there is a smooth Ricci flow solution
  $(M, g(\cdot))$ with $g(0) = g_0$ that exists for all $t \ge 0$. There
  is a constant $C < \infty$ so that $\parallel \Rm(g(t)) \parallel_\infty \le
  \frac{C}{t}$ for all $t \ge 0$.
\end{proposition}

The main technical result of this paper is that a three dimensional
Ricci flow solution
$(M,g(t))$ with positive Ricci curvature, that satisfies the conclusion
of Proposition
\ref{1.5}, admits a three-dimensional blowdown limit.

\begin{proposition} \label{1.6}
  Let $(M, g_0, m_0)$ be a complete connected pointed
  Riemannian manifold of dimension
  three, with bounded sectional curvature and positive Ricci curvature.
  Suppose that the ensuing Ricci flow exists for all $t \ge 0$, and that
  there is some $C < \infty$ so that $\parallel \Rm(g(t)) \parallel_\infty \le
  \frac{C}{t}$ for all $t \ge 0$.
  For $s > 0$, put $g_s(t) = s^{-1} g(st)$.
  Then for some sequence $s_i \rightarrow \infty$,
  there is a limit $\lim_{i \rightarrow \infty} g_{s_i}(\cdot) =
  g_\infty(\cdot)$ in the pointed Cheeger-Hamilton topology.
  The Ricci flow solution $g_\infty(u)$ lives on a three dimensional manifold
  and is defined for $u > 0$.
\end{proposition}

The issue in proving Proposition \ref{1.6} is to rule out collapsing
at large time.
Examples of Proposition \ref{1.6} come from expanding gradient solitons,
for which the tangent cone at infinity can be the cone over any
two-sphere with Gaussian curvature greater than one
\cite{Deruelle (2016)}.  Of course, these are
not $c$-Ricci pinched (Lemma \ref{4.6}).

Using distance distortion estimates, Proposition \ref{1.6} has the following
implication about the initial metric.

\begin{corollary} \label{1.7}
  Under the hypotheses of Proposition \ref{1.6},
  the Riemannian manifold $(M, g_0)$ has cubic volume growth.
  \end{corollary}

The proof of Theorem \ref{1.3}(a) then uses a Ricci flow result of Simon-Schulze
\cite{Schulze-Simon (2013)}. To prove Theorem \ref{1.3}(b) we apply a
spatial
rescaling argument to $(M,g)$.

The proof of Theorem \ref{1.4} uses Corollary \ref{1.7} and 
results of \cite{Lebedeva-Petrunin} about weak
convergence of curvature operators. Assuming that $M$ is noncompact,
we apply a spatial rescaling to
$(M, g)$ to get an locally Alexandrov
three dimensional tangent cone at
infinity.  If $(M, g)$ is nonflat then
the weak convergence of curvature operators, along with
the $c$-Ricci
pinching, forces the
tangent cone at infinity of $(M, g)$ to be $\R^3$, which contradicts the
nonflatness assumption.

One could ask about generalizations of
Conjectures \ref{1.1} and \ref{1.2} without the uniform
curvature bound, or in higher dimension.
However, in this paper we stick with three dimensions and bounded
sectional curvature. The paper \cite{Ni-Wu (2007)} shows compactness in general dimension
for complete Riemannian manifolds that have bounded sectional curvature and pointwise pinched positive
curvature operator. Having a positive curvature operator allows the papers
\cite{Chen-Zhu (2000),Ni-Wu (2007)}
to apply the
Gromoll-Meyer injectivity radius bound and Hamilton's differential Harnack inequality.
As we do not make an assumption of positive curvature operator, we have to develop
different tools.

The structure of the paper is the following.  In Section \ref{sect2} we prove
Proposition \ref{1.5} and give some distance distortion estimates.
In Section \ref{sect3} we prove Proposition \ref{1.6}.  Section \ref{sect4} has the
proof of Corollary \ref{1.7}.  In Section \ref{sect5} we prove Theorem \ref{1.3} and
in Section \ref{sect6} we prove Theorem \ref{1.4}.  More detailed descriptions
are at the beginnings of the sections.

I thank Nina Lebedeva for sending me a preliminary version of
\cite{Lebedeva-Petrunin}.
I also thank Christoph B\"ohm, Simon Brendle, Yi Lai, Man-Chun Lee, Yu Li and the referee
for comments on an earlier version of this paper.

Note: After this paper was submitted for publication, Conjecture \ref{1.1}
was proved by Deruelle, Schulze and Simon
\cite{Deruelle-Schulze-Simon (2022)}.  The proof uses Propositions
\ref{1.5} and \ref{1.6} of the present paper.  The assumption of bounded
sectional curvature was then removed by Lee and Topping \cite{Lee-Topping (2022)}.

\section{Long-time existence and curvature decay} \label{sect2}

In this section we prove Proposition \ref{1.5}.  
We first show that the Ricci flow exists for all $t > 0$. The proof is similar
to an argument in Hamilton's original Ricci flow paper
\cite{Hamilton (1982)} about what could possibly happen at a curvature
blowup under the Ricci pinching assumption.  When applied to long-time
solutions, essentially the same argument is used to rule out type-II
solutions, thereby proving the curvature bound in Proposition \ref{1.5}.
Using the curvature bound, we give some distance distortion estimates
that will be important in Section \ref{sect3}.

We begin by recalling some facts about Ricci flow.
Let $(M, g_0)$ be a Riemannian
manifold as in the statement of Conjecture \ref{1.1}.
Let $(M, g(\cdot))$
denote the unique maximal Ricci flow solution with
initial time slice $g(0) = g_0$,
having complete time slices and bounded curvature on compact time
intervals.
The condition $\Ric \ge 0$ is preserved under Ricci flow.
Using the weak maximum principle, one can show that
being $c$-Ricci pinched is preserved under Ricci flow.
Using the strong maximum principle, if $(M, g_0)$ is nonflat
then for $t > 0$, the Ricci curvature is positive.
Hence we can assume that $(M, g_0)$ has positive Ricci curvature.
This shows the equivalence between Conjecture \ref{1.1} and Conjecture \ref{1.2}.

Under the hypotheses of Conjecture \ref{1.2}, to argue by contradiction,
hereafter we also assume that $M$ is noncompact.
Then it is diffeomorphic to $\R^3$ \cite{Schoen-Yau (1982)}.

\begin{proposition} \label{2.1}
  The Ricci flow solution $(M, g(\cdot))$ exists for all $t \ge 0$.
\end{proposition}
\begin{proof}
  We have
\begin{equation} \label{2.2}
  r_1 \ge c r_3 \Rightarrow r_1 \ge \frac12 c (r_2 + r_3) \Rightarrow
  \left( 1 + \frac12 c \right) r_1 \ge \frac12 c(r_1 + r_2 + r_3).
\end{equation}
Hence $\Ric \ge \rho R$, where $\rho = \frac{c}{2+c} \in \left( 0, \frac13
\right]$, and $R$ denotes the scalar curvature. Put $\sigma = \rho^2$.

  Suppose that the maximal Ricci flow solution is on a finite time interval
  $[0, T)$.
  We claim first that for all $t \in [0,T)$, we have
  \begin{equation} \label{2.3}
    R^{\sigma-2} \left| \Ric - \frac13 R g(t) \right|^2 \le
    \left( \frac{3}{2t} \right)^\sigma
  \end{equation}
  everywhere on $M$.
To prove this, we combine methods from
  \cite[Pf. of Proposition 3]{Brendle-Huisken-Sinestrari (2011)} and
  \cite[Pf of Lemma 6.1]{Chen-Zhu (2000)}. Put
  \begin{equation} \label{2.4}
f =  R^{\sigma-2} \left| \Ric - \frac13 R g(t) \right|^2.
    \end{equation}
  From the bounded curvature assumption, $f$ is uniformly bounded above
  at time zero.
  From \cite[p. 539]{Brendle-Huisken-Sinestrari (2011)} and
  \cite[Eqn. (76)]{Chen-Zhu (2000)}, which are based on
  \cite[Lemma 10.5]{Hamilton (1982)},
  \begin{equation} \label{2.5}
    \left( \frac{\partial}{\partial t} - \triangle \right) f \le
    2(1-\sigma) \left\langle \frac{\nabla R}{R}, \nabla f \right\rangle -
    \sigma (1-\sigma) R^{\sigma-4} \left| \Ric - \frac13 R g(t) \right|^2
    |\nabla R|^2 - \frac23 \sigma f^{1+\frac{1}{\sigma}}.
  \end{equation}
  If $M$ were compact then we could immediately derive (\ref{2.3}) using the
  weak maximum principle, as in
  \cite[Proposition 3]{Brendle-Huisken-Sinestrari (2011)}. If $M$ is noncompact
  then the possible unboundedness of $\frac{\nabla R}{R}$ is an issue.
  To get around this, using
  \begin{equation} \label{2.6}
    2 \left\langle \frac{\nabla R}{R}, \nabla f \right\rangle \le
    \sigma f \left| \frac{\nabla R}{R} \right|^2 +
    \frac{|\nabla f|^2}{\sigma f},
  \end{equation}
  we obtain
  \begin{equation} \label{2.7}
    \left( \frac{\partial}{\partial t} - \triangle \right) f \le
\frac{1-\sigma}{\sigma} \frac{|\nabla f|^2}{f} - 
\frac23 \sigma f^{1+\frac{1}{\sigma}}.
  \end{equation}
  Equivalently,
  \begin{equation} \label{2.8}
    \left( \frac{\partial}{\partial t} - \triangle \right)
    f^{\frac{1}{\sigma}} \le
- \: \frac23 f^{\frac{2}{\sigma}}
  \end{equation}
    in the barrier sense.
  From the weak maximum principle,
  \begin{equation} \label{2.9}
    \sup_{m \in M} f^{\frac{1}{\sigma}}(m,t) \le \frac{3}{2t},
  \end{equation}
  which proves the claim.

  There is a sequence $\{t_i\}_{i=1}^\infty$ of times
    increasing to $T$, and points $\{m_i\}_{i=1}^\infty$ in $M$ so
    that $\lim_{i \rightarrow \infty} |\Rm(m_i, t_i)| = \infty$ and
    $|\Rm(m_i, t_i)| \ge \frac12 \sup_{(m,t) \in M \times [0,t_i]}
    |\Rm(m, t)|$.
    Put $Q_i = |\Rm(m_i, t_i)|$ and
    $g_i(x,u) = Q_i g(x, t_i + Q_i^{-1} u)$. Then
    $g_i$ is a Ricci flow solution with curvature norm equal to one at $(m_i, 0)$, and curvature
    norm uniformly bounded above by two for
    $u \in [-Q_i t_i, 0]$.

    Suppose first that for some $i_0 > 0$ and all $i$,
    we have $Q_i \inj_{g(t_i)}(m_i)^2 \ge i_0$.
    (This does not follow from Perelman's no local collapsing
    result,
    since we do not assume that the initial metric has positive
    injectivity radius.)
    After passing to
    a subsequence, there is a pointed Cheeger-Hamilton limit
    \begin{equation} \label{2.10}
    \lim_{i \rightarrow \infty} (M, g_i(\cdot), m_i) = 
    (M_\infty, g_\infty(\cdot), m_\infty),
    \end{equation}
    where
    $g_\infty(u)$ is defined for $u \in (- \infty, 0]$.
  The property of having nonnegative Ricci curvature passes to the limit.
By construction, $g_\infty$ has curvature norm one at $(m_\infty, 0)$.
Hence $g_\infty$ has positive scalar curvature at $(m_\infty, 0)$. By the
strong maximum principle, it follows that $g_\infty$ has positive scalar
curvature everywhere.

Given $m^\prime \in M_\infty$,
the point
$(m^\prime, 0)$ is the limit of a sequence of points
$\{(m_i^\prime, 0)\}_{i=1}^\infty$ with
$\lim_{i \rightarrow \infty} 
R_{g_i}( m_i^\prime, 0) = R_{g_\infty}(m^\prime, 0) > 0$.
As $\lim_{i \rightarrow \infty} Q_i = \infty$, after undoing the rescaling 
it follows that $\lim_{i \rightarrow \infty} 
R_{g}( m_i^\prime, t_i) = \infty$. As
$\lim_{i \rightarrow \infty} t_i = T$, we also have
$\lim_{i \rightarrow \infty} 
t_i R_{g}( m_i^\prime, t_i) = \infty$.
Applying 
(\ref{2.3}) to $g_i$ and taking the limit as $i \rightarrow \infty$,
it follows that the metric $g_\infty(0)$ satisfies
  $\Ric - \frac13 R g_\infty(0) = 0$. As $g_\infty(0)$ has
  positive scalar curvature at $(m_\infty, 0)$, it follows that
  $M_\infty$ is a spherical space form. Then $M$ is compact,
  which is a contradiction.

    Even if there is no uniform positive lower 
 bound for $Q_i \inj_{g(t_i)}(m_i)^2$,
    after passing to
    a subsequence, there is a pointed limit
    \begin{equation} \label{2.11}
    \lim_{i \rightarrow \infty} (M, g_i(\cdot), m_i) = 
    ({\mathcal G}_\infty, g_\infty(\cdot), {\mathcal O}_{x_\infty}).
    \end{equation}
    Here ${\mathcal G}_\infty$ is a three dimensional closed
    Hausdorff \'etale groupoid and $g_\infty(\cdot)$ is a family of invariant
    Riemannian metrics on the unit space of ${\mathcal G}_\infty$
    \cite[Section 5]{Lott (2007)}.
    Let $X_\infty$ denote
    he orbit space of ${\mathcal G}_\infty$; then
    ${\mathcal O}_{x_\infty} \in X_\infty$ is a basepoint.
    The Ricci flow
    $g_\infty(u)$ is defined for $u \in (- \infty, 0]$.
  For each $u$, the metric $g_\infty(u)$ induces a
  metric on $X_\infty$ that makes it into
  a complete metric space.
  As before, $\lim_{i \rightarrow \infty} R_g(m_i, t_i) = \infty$ and
 (\ref{2.3}) again implies that the metric $g_\infty(0)$ satisfies
  $\Ric - \frac13 R g_\infty(0) = 0$. As $g_\infty(0)$ has
  positive scalar curvature along the orbit ${\mathcal O}_{x_\infty}$
  in the unit space, the metric
  $g_\infty(0)$ has constant positive Ricci curvature.
  The argument for the Bonnet-Myers theorem implies that
  $X_\infty$ is compact; c.f. \cite[Section 2.9]{Hilaire (2014)}.
  Then $M$ is compact,
  which is a contradiction.
\end{proof}

\begin{remark} \label{2.12}
  One could avoid the use of \'etale groupoids by first looking at the
  pullback flows on $T_{m_i} M$ and taking a limit, to argue that
    for large $i$, the metric
    $g(t_i)$ has almost constant positive sectional curvature on
    $B \left( m_i, R(m_i, t_i)^{- \: \frac12} \right)$. One could then
    shift basepoints and repeat the argument, to obtain that for any
    $A < \infty$ and for large $i$, the metric
    $g(t_i)$ has almost constant positive sectional curvature on
    $B \left( m_i, A R(m_i, t_i)^{- \: \frac12} \right)$. From
    Bonnet-Myers, one concludes that $M$ is compact, which is a
    contradiction.
  \end{remark}

\begin{proposition} \label{2.13}
    There is some $C < \infty$ so that for all $t \in [0, \infty)$, we have
    $\|\Rm(g(t)) \|_\infty \le \frac{C}{t}$.
    \end{proposition}
\begin{proof}
  Suppose that the proposition is not true. After doing a type-II
  point picking
  \cite[Chapter 8, Section 2.1.3]{Chow-Lu-Ni (2006)}, there are
  points $(m_i, t_i)$ so that $\lim_{i \rightarrow \infty}
  t_i |\Rm(m_i, t_i)| = \infty$ and $|\Rm| \le 2 |\Rm(m_i, t_i)|$
  on $M \times [a_i, b_i]$, with $\lim_{i \rightarrow \infty}
  |\Rm(m_i, t_i)| (t_i - a_i) = \lim_{i \rightarrow \infty}
  |\Rm(m_i, t_i)| (b_i - t_i) = \infty$.
    Put $Q_i = |\Rm(m_i, t_i)|$ and
    $g_i(x,u) = Q_i g(x, t_i + Q_i^{-1} u)$.

  Suppose first that for
  some $i_0 > 0$ and all $i$, we have
  $Q_i \inj_{g(t_i)}(m_i)^2 \ge i_0$. After passing to a subsequence,
we get a limiting Ricci flow solution
$\lim_{i \rightarrow \infty} \left( M, g_i(\cdot), m_i \right) =
\left( M_\infty, g_\infty(\cdot), m_\infty \right)$ defined for
times $u \in \R$. Here $M_\infty$ is a $3$-manifold and 
$|\Rm(m_\infty, 0)| = 1$. As in the proof of Proposition \ref{2.1},
for each $m^\prime \in M_\infty$, the point
$(m^\prime, 0)$ is the limit of a sequence of points $(m_i^\prime, 0)$
with $\lim_{i \rightarrow \infty} t_i R_{g}(m_i^\prime, t_i) = \infty$,
where the latter statement now comes from the type-II rescaling.
From (\ref{2.3}), we get
  $\Ric - \frac13 R g_\infty = 0$. Then $(M_\infty, g_\infty)$
  has constant positive curvature time slices, which implies that
  $M_\infty$ is compact.  Then $M$ is also compact, which is a
  contradiction.

  If $\liminf_{i \rightarrow \infty} Q_i \inj_{g(t_i)}(m_i)^2 = 0$,
  we can still take a limit as in (\ref{2.11}). As in the argument after
  (\ref{2.11}), we again conclude that $M$ is compact, which is a
  contradiction.
    \end{proof}

\begin{corollary} \label{2.14}
  There are numbers $\{A_k\}_{k=0}^\infty$
  that for all $t \in [0, \infty)$ and all multi-indices $I$, we have
    $\| \nabla^I \Rm \|_{g(t)} \le A_{|I|} t^{- \frac{|I|}{2} - 1}$.
\end{corollary}
\begin{proof}
  This follows from Proposition \ref{2.13}, along with derivative estimates for
the Ricci flow \cite[Theorem 6.9]{Chow-Lu-Ni (2006)}.
  \end{proof}

Let $d_t : M \times M \rightarrow \R$ be the distance function on $M$
with respect to the Riemannian metric $g(t)$. In particular, $d_0$ is the
distance function with respect to $g_0$.

\begin{lemma} \label{2.15}
  There is some $C^\prime < \infty$ so that whenever
  $0 \le t_1 \le t_2 < \infty$, we have
  \begin{equation} \label{2.16}
    d_{t_1} - C^\prime \left( \sqrt{t_2} - \sqrt{t_1} \right)
      \le d_{t_2} \le d_{t_1}.
    \end{equation}
  \end{lemma}
\begin{proof}
  This follows from distance distortion estimates for Ricci flow, as
  in \cite[Remark 27.5 and Corollary 27.16]{Kleiner-Lott (2008)}.
\end{proof}

Fix $m_0 \in M$. Given $s > 0$, put $g_s(u) = s^{-1} g(su)$. Then
$(M, g_s(\cdot))$ is also a Ricci flow solution, with
$\parallel \Rm(g_s(u)) \parallel \le \frac{C}{u}$ and
$\| \nabla^I \Rm \|_{g_s(u)} \le A_{|I|} u^{- \frac{|I|}{2} - 1}$.
Its distance function at time $u$ is
$\widehat{d}_{s,u} = s^{- \: \frac12} d_{su}$. From (\ref{2.16}), we have
\begin{equation} \label{2.17}
  \frac{1}{\sqrt{s}} d_0 - C^\prime \sqrt{u} \le \widehat{d}_{s,u} \le
  \frac{1}{\sqrt{s}} d_0.
  \end{equation}
Given $\rho > 0$, it follows that
\begin{equation} \label{2.18}
  B_{\widehat{d}_{s,u}}(m_0, \rho - C^\prime \sqrt{u})
  \subset B_{d_0}(m_0, \rho \sqrt{s}) \subset
  B_{\widehat{d}_{s,u}}(m_0, \rho)
  \end{equation}

Also, if
$0 \le s_1 \le s_2 < \infty$ then
  \begin{equation} \label{2.19}
    \sqrt{\frac{s_1}{s_2}} \widehat{d}_{s_1,u} - C^\prime
    \left( 1 - \sqrt{\frac{s_1}{s_2}}
    \right) \sqrt{u}
   \le \widehat{d}_{s_2,u} \le \sqrt{\frac{s_1}{s_2}} \widehat{d}_{s_1,u}.
    \end{equation}
Given $\rho > 0$, it follows that
\begin{equation} \label{2.20}
B_{\widehat{d}_{s_{2},u}} \left( m_0,
\sqrt{\frac{s_1}{s_2}}
\rho - C^\prime \left( 1 - \sqrt{\frac{s_1}{s_2}}
\right) \sqrt{u} \right) \subset
B_{\widehat{d}_{s_{1},u}} \left( m_0,
\rho \right)
   \subset B_{\widehat{d}_{s_{2},u}} \left( m_0, \sqrt{\frac{s_1}{s_2}}
\rho \right).
\end{equation}

  Given a sequence
  $\{s_i\}_{i=1}^\infty$ tending to infinity and $u > 0$,
  after passing to a subsequence
  we can assume that there is a limit
  of $\lim_{i \rightarrow \infty} (M, g_{s_i}(u), m_0)$
  in the pointed Gromov-Hausdorff
  topology. We claim that we can choose the subsequence so that
  the limit exists simultaneously for each $u$,
  and as $u$ varies the limiting metric
  spaces are all biLipschitz equivalent to each other.  To see this,
after passing to a subsequence we can assume that there is a limit  
$\lim_{i \rightarrow \infty} (M, g_{s_i}(\cdot), m_0) =
({\mathcal G}_\infty, g_\infty(\cdot), {\mathcal O}_{x_\infty})$.
Here $g_\infty(\cdot)$ is a Ricci flow solution on the \'etale
groupoid ${\mathcal G}_\infty$, that exists for $u > 0$.
As $u$ varies,
the pointed Gromov-Hausdorff limit
$\lim_{i \rightarrow \infty} (M, g_{s_i}(u), m_0)$ always has the
same underlying pointed topological space, namely the
pointed orbit space
$(X_\infty, x_\infty)$ of
${\mathcal G}_\infty$.  The metric on the limit depends on $u$, and
is the quotient metric
$\widehat{d}_{\infty,u}$ coming from $g_\infty(u)$. It follows that
the various quotient metrics, as $u$ varies, are biLipschitz to each
other.

Since $M$ is noncompact, $X_\infty$ is also noncompact. In particular,
$\dim(X_\infty) > 0$.

\section{Noncollapsing at large time} \label{sect3}

In this section we show that the Ricci flow solution
from Section \ref{sect2} is noncollapsed for large time, in a scale-invariant
sense.  More precisely, we show that there is a blowdown limit on
a three dimensional manifold, where the emphasis is on the three
dimensionality.

We recall that the Ricci flow solution from Section \ref{sect2} has positive
Ricci curvature and lives on a noncompact manifold, which is necessarily
then diffeomorphic to $\R^3$. After passing to a subsequence, we can
extract a blowdown limit $X_\infty$ (corresponding to a fixed rescaled time)
in the sense of pointed Gromov-Hausdorff convergence. The issue is to
show that $\dim(X_\infty) = 3$. Since $X_\infty$ is noncompact, we must
exclude that $\dim(X_\infty)$ is one or two.
We note that $\R^3$ 
can collapse with bounded
sectional curvature \cite[Example 1.4]{Cheeger-Gromov (1985)} due to a
graph manifold structure, so the result is not immediate. 

The following statement is the main result of this section.

\begin{proposition} \label{3.1}
  There is some sequence $\{s_i\}_{i=1}^\infty$ tending to infinity
  so that the pointed limit
  $\lim_{i \rightarrow \infty} (M, g_{s_i}(\cdot), m_0)$ exists as a
  Ricci flow $(M_\infty, g_\infty(\cdot), m_\infty)$ on a
  pointed $3$-manifold $(M_\infty, m_\infty)$. 
  \end{proposition}

  Before giving the details of the proof, we sketch the main ideas.
  Suppose that the proposition is not true.  Fixing $u$, for large $s_0$
  the metric space $(M, {\widehat{d}}_{s_0,u}, m_0)$ is pointed
  Gromov-Hausdorff close to a noncompact Alexandrov space of dimension
  one or two.  In either case, there is a short loop
  $\gamma$ at $m_0$ that is
  not contractible in ${B_{\widehat{d}_{s_0,u}}(m_0, 1)}$. Since $M$ is
  contractible, $\gamma$ can be contracted in
  ${B_{\widehat{d}_{s_0,u}}(m_0, \Delta)}$ for some $\Delta > 1$.
  From the distance shrinking in (\ref{2.19}), there will be some $s_1 \ge s_0$ so that
$\gamma$ cannot be contracted in
${B_{\widehat{d}_{s_1,u}}(m_0, 1)}$ but can be contracted in
${B_{\widehat{d}_{s,u}}(m_0, 1)}$ for all $s > s_1$.
From continuity and the
definition of $s_1$, the loop $\gamma$ can
be contracted in 
${B_{\widehat{d}_{s_1,u}}(m_0, 2)}$ but cannot be contracted in
${B_{\widehat{d}_{s_1,u}}(m_0, 1/10)}$.
When changing metrics from
$\widehat{d}_{s_0,u}$ to $\widehat{d}_{s_1,u}$, the length of $\gamma$
can only go down.  Then $(M, {\widehat{d}}_{s_1,u}, m_0)$ will be
pointed Gromov-Hausdorff close to a ray, with $m_0$
corresponding to a point of distance approximately one from the tip of the
ray. The part of $M$ in which $\gamma$ contracts is a very
collapsed solid torus. From the geometry of such solid tori,
if a very short loop at $m_0$ is
contractible in ${B_{\widehat{d}_{s_1,u}}(m_0, 2)}$ then it is
already contractible in
${B_{\widehat{d}_{s_1,u}}(m_0, 1/10)}$, which is a contradiction.

In what follows, $\const$ will denote a constant that is independent of the other parameters in the statement. 
Given a sequence $\{K_i\}_{i=0}^\infty$, we will say that a Riemannian manifold is
$\vec{K}$-regular if $\| \nabla^I \Rm \| \le K_{|I|}$ for all multi-indices $I$. For
simplicity, we will take $\| \Rm \| \le K_0$ to mean that the sectional curvatures are
bounded by $K_0$ in magnitude.
We will say that a loop $\gamma$ at a basepoint $m_0$ is contractible in a ball
$B(m_0, r)$ if it can be
contracted to a point by a family of loops in $B(m_0, r)$ that go through $m_0$, i.e. that
$\gamma$ represents a trivial element of $\pi_1(B(m_0, r), m_0)$.
We begin by proving some lemmas about collapsed $3$-manifolds.
The first one is relevant when the manifold is pointed Gromov-Hausdorff close to a
one dimensional space. (See Figure
\ref{figure1}.)
\begin{figure}[ht!]
\includegraphics[width=6in]{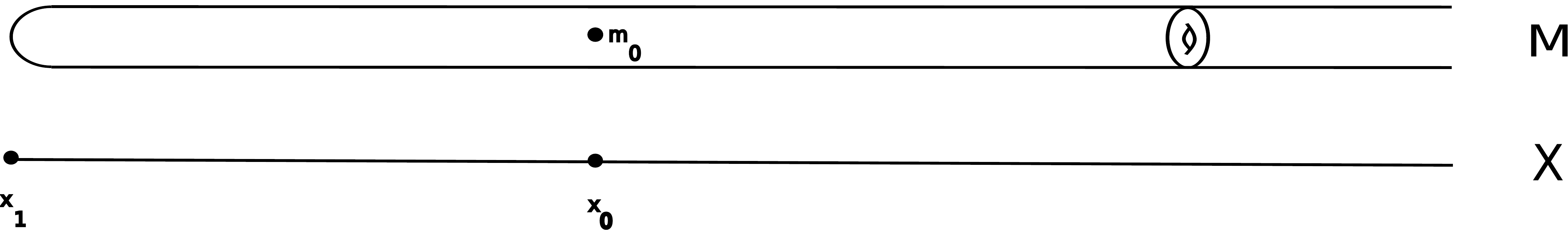}
\caption{}
\label{figure1}
\end{figure}

\begin{lemma} \label{gl1}
  Given $\vec{K}$ and $L < \infty$, there is some $\overline{\delta_1} = \overline{\delta_1}(\vec{K},L) > 0$ with the following
  property.  Suppose that $(M, m_0)$ is a complete pointed $\vec{K}$-regular oriented
  Riemannian
  $3$-manifold
  so that
  $(M, m_0)$ has pointed Gromov-Hausdorff distance
  $\delta_1 \le \overline{\delta_1}$
  from a line or ray, which we denote by $(X, x_0)$.
  Then  there is a loop $\gamma$ through $m_0$ with length
  $l(\gamma) \le \const_1 \delta_1$ so that $\gamma$ is not contractible in
  ${B(m_0, L)}$.
\end{lemma}
\begin{proof}
It is enough to consider ${\delta_1}$'s that are much smaller than $L^{-1}$.
We apply \cite[Theorem 1.7]{Cheeger-Fukaya-Gromov (1992)}.
In our case, the relevant nilpotent Lie groups $N$ to describe the
local geometry near a point $m \in M$, from
\cite[p. 331]{Cheeger-Fukaya-Gromov (1992)},
are $\R^2$ if $X$ is a line and $\R \times U(1)$ if $X$ is a ray.

If $X$ is a line and $\delta_1$ is 
sufficiently small then the pointed Gromov-Hausdorff approximation can be
realized by a map that restricts to a fibration
$\pi : U \rightarrow ( - \delta_1^{-1}/2, \delta_1^{-1}/2)$, where $U$ is an open subset of
$M$, $\pi(m_0) = 0 = x_0$,
the fibers of $\pi$ have diameter at most $\const \delta_1$, and the fibers are $2$-tori
or Klein bottles.  Since an oriented $3$-manifold cannot contain a two-sided Klein bottle, the fibers must be $2$-tori.
Up to a small biLipschitz perturbation, we can assume that the fibers are flat
\cite[Theorem 1.3]{Cheeger-Fukaya-Gromov (1992)}.  We can then take
$\gamma$ to be a shortest nontrivial closed geodesic in $\pi^{-1}(x_0)$ passing through
$m_0$.

If $X$ is a ray, let $x_1$ denote its tip.  If $d(x_0, x_1) > 2L$ then we are effectively in
the previous case.  If $d(x_0, x_1) \le 2L$ then
there is a singular fibration $\pi : U \rightarrow B(x_0, \delta_1^{-1}/2)$, where $U$ is an open subset of
$M$ and the preimages of $\pi$ have diameter at most $\const \delta_1$. If $x \neq x_1$ then
$\pi^{-1}(x)$ is a $2$-torus, while $\pi^{-1}(x_1)$ is a circle. 
From \cite[Theorem 1.7]{Cheeger-Fukaya-Gromov (1992)}, a local biLipschitz
model around 
$\pi^{-1}(x_1)$ is $(\R \times B^2)/\Lambda$, where $\R \times B^2$ has a metric that is
$\R \times U(1)$-invariant and $\Lambda$ is a discrete subgroup of $\R \times U(1)$ that
is isomorphic to $\Z$. More precisely, for any given $\epsilon > 0$, there are $\rho > 0$ and
$k \in \Z^+$ so that there is a $(\rho, k)$-round model
metric, in the sense of \cite[Section 1]{Cheeger-Fukaya-Gromov (1992)}, that is $e^{\epsilon}$-biLipschitz 
close to the geometry near $\pi^{-1}(x_1)$; in our case the model metric takes the form
$(\R \times B^2)/\Lambda$.
To be close to a ray,
a generator $(\tau, e^{i \phi})$ of $\Lambda$ has $\phi$ a small nonzero number and 
$\tau$ very small,
relative to $\phi$. Topologically, $(\R \times B^2)/\Lambda$ is a solid torus, as can
be seen by isotoping $\phi$ to zero.  Away from this model around $\pi^{-1}(x_1)$, there is a topological product structure of an interval with a $2$-torus, so 
$\pi^{-1}(B(x_0, 3L))$ is a solid torus.

As before, up to a small biLipschitz perturbation, we can assume that the generic preimage
of $\pi$ is a flat $2$-torus.  If $x_0 = x_1$ then we can take $\gamma$ to be the circle
$\pi^{-1}(x_0)$.  If $x_0 \neq x_1$, let $\gamma_0$ be a shortest nontrivial closed geodesic
in the $2$-torus $\pi^{-1}(x_0)$ passing through $m_0$ and let $\gamma_1$ be a shortest closed geodesic in $\pi^{-1}(x_0)$ passing through 
$m_0$ that is not homotopic to a multiple of $\gamma_0$. They both have length at most
$\const_1 \delta_1$. At least one of $\gamma_0$ and
$\gamma_1$ is noncontractible in the solid torus $\pi^{-1}(B(x_0, 3L))$ and we take
$\gamma$ to be that curve.
\end{proof}

Given $K, \delta_1 > 0$, let $A_{K, \delta_1}$ be the set of pointed two dimensional complete
Alexandrov spaces $(X,d_X,\star_X)$ with Alexandrov curvature bounded below by $- 2K$
that have pointed Gromov-Hausdorff distance at least $\delta_1/2$ from any pointed one
dimensional complete Alexandrov space. It is compact in the pointed Gromov-Hausdorff
topology.

The next lemma is relevant when the manifold is pointed Gromov-Hausdorff close to a
two dimensional space but is not pointed Gromov-Hausdorff close to a one dimensional
space.

\begin{lemma} \label{gl2}
  Given $L < \infty$, $\delta_1 > 0$ and $\vec{K}$, there is some $\overline{\delta_2} =
  \overline{\delta_2}(\vec{K},L, \delta_1) > 0$ with the following
  property.  Suppose that $(M, m_0)$ is a complete pointed $\vec{K}$-regular
  Riemannian
  $3$-manifold diffeomorphic to $\R^3$
  so that
  $(M, m_0)$ has pointed Gromov-Hausdorff distance
  $\delta_2 \le \overline{\delta_2}$
  from an element $(X, x_0)$ of $A_{K_0,\delta_1}$.
  Then  there is a loop $\gamma$ through $m_0$ with length
  $l(\gamma) \le \const_2 \delta_2$ so that $\gamma$ is not contractible in
  ${B(m_0, L)}$.
\end{lemma}
\begin{proof}
  We can assume that $\delta_2^{-1} \gg L$. We can apply \cite[Theorem 1.7]{Cheeger-Fukaya-Gromov (1992)}, with the relevant nilpotent Lie group $N$ being $\R$. The conclusion is that
  if $\delta_2$ is sufficiently small then the pointed Gromov-Hausdorff
  approximation can be realized by a map that restricts to an orbifold circle fibration 
  $\pi : U \rightarrow B(x_0, \delta_2^{-1}/2)$ where $U$ is an open subset of $M$ and the
  fibers of $\pi$ have diameter at most $\const_2 \delta_2$.  
  (As $A_{K_0, \delta_1}$ is compact, we can choose
  $\delta_2$ independent of $(X, x_0) \in A_{K_0, \delta_1}$.) As $M$ is orientable,
  the orbifold $B(x_0, \delta_2^{-1}/2)$ has isolated singular points, i.e. there are
  no reflector lines.  Since $M$ does not have any embedded Klein bottles, the
  circle fibration is orientable and so describes a Seifert fibration.  As 
  $B(x_0, 2L)$ is noncompact, from
  \cite[Lemma 3.2]{Scott (1983)}
  there is an exact sequence
  \begin{equation} \label{exact}
  1 \rightarrow \Z \rightarrow
  \pi_1
  \left( \pi^{-1}(B(x_0, 2L)) \right)
  \rightarrow
  \pi_1
  \left( B(x_0, 2L) \right) \rightarrow 1,
  \end{equation}
  where the image of a generator of $\Z$ is represented by a regular fiber
  of the Seifert fibration, and
  $\pi_1
  \left( B(x_0, 2L) \right)$ denotes the
  orbifold fundamental group.
  (Since the $\Z$-subgroup is central in $\pi_1(\pi^{-1}(B(x_0, 2L)))$,
  it is well-defined independent of
  basepoint.) If $\delta_2$ is small then
  ${B(m_0, L)} \subset \pi^{-1}(B(x_0, 2L))$ and 
  we can take $\gamma = \pi^{-1}(x_0)$.
\end{proof}

The next lemma describes the geometry of a collapsed space with a short loop at the
basepoint that can be contracted in a ball of radius $2L$ around the basepoint, but
not in a ball of radius $L/10$ around the basepoint. 

\begin{lemma} \label{gl3}
  Given $L < \infty$, $\delta_3 \ll L^{-1}$ and $\vec{K}$, there is some $\delta_4 =
  \delta_4(\vec{K},L, \delta_3) > 0$ with the following
  property.  Suppose that $(M, m_0)$ is a complete noncompact pointed $\vec{K}$-regular
  oriented Riemannian
  $3$-manifold
  so that
  $(M, m_0)$ has pointed Gromov-Hausdorff distance at most
  $\delta_4$ from a one or two dimensional complete pointed Alexandrov space with Alexandrov curvature bounded below by $-2K_0$. Suppose that there is a loop $\gamma$ through $m_0$ of length
  $l(\gamma) < \delta_4$ that can be contracted in ${B(m_0, 2L)}$ but not in 
  ${B(m_0, L/10)}$. Then
   $(M, m_0)$ has pointed Gromov-Hausdorff distance at most $\delta_3$ from a ray $(X, x_0)$.
\end{lemma}
\begin{proof}
Given $\delta_1 > 0$, which we will adjust, as in the proof of Lemma \ref{gl2} if $\delta_4$ is 
sufficiently small and $(M, m_0)$ is $\delta_4$-close to some
$(X, x_0) \in 
A_{K_0, \delta_1}$ then the pointed Gromov-Hausdorff approximation can be
realized by an orbifold circle fibration
$\pi : U \rightarrow B(x_0, \delta_4^{-1})$.  

We claim that
there is an $r = r(\vec{K}, \delta_1) \ll L/100$ so that the pointed Gromov-Hausdorff
approximation can be chosen such that 
$B(x_0, r)$ has at most one singular point.
To see this, suppose by way of contradiction that there are sequences $\{ M_i \}_{i=1}^\infty$ and $\{ \delta_{4,i} \}_{i=1}^\infty$ with $\lim_{i \rightarrow \infty} \delta_{4,i} = 0$
so that $\lim_{i \rightarrow \infty} M_i = X_\infty$ for some 
$(X_\infty, x_{0, \infty}) \in A_{K_0, \delta_1}$, but if $\pi_i : M_i \rightarrow X_i$ is 
an orbifold circle bundle that is a pointed $\delta_{4,i}$-approximation to
an $(X_i, x_{0,i}) \in A_{K_0, \delta_1}$ then $B(x_{0,i}, 1/i)$ has more than one
orbifold point.  After passing to a subsequence, the frame bundles 
$\{ FM_i \}_{i=1}^\infty$ converge in the pointed smooth $\SO(3)$-equivariant topology to a
$5$-manifold ${\mathcal M}$ with a locally free $\SO(3)$-action, and $X_\infty = 
{\mathcal M}/\SO(3)$
\cite[Proposition 11.5 and Theorem 12.8]{Fukaya (1990)}.
As each $M_i$ is orientable, $X_\infty$ has isolated singular points.
For large $i$ there is a pointed $\delta_{4,i}$-approximation
$\pi_i : M_i \rightarrow X_\infty$, which is a contradiction.

Then
$\pi^{-1}(B(x_0, r))$ is a solid torus.  As $\gamma$ is a short
loop in $\pi^{-1}(B(x_0, r))$ that is not contractible in ${B(m_0, 2r)}$, it is homotopic to
a nonzero multiple of a circle fiber of the fibration.  Then from (\ref{exact}), the loop
$\gamma$ is not contractible in ${B(m_0, 2L)}$, which is a
contradiction. Hence $(M, m_0)$ isn't $\delta_4$-close to an element of
$A_{K_0, \delta_1}$. In particular, if $\delta_4$ is small enough relative to $\delta_1$
then
$(M, m_0)$ is $\delta_1$-close to a pointed line or ray $(X, x_0)$.

If $(X,x_0)$ is a line and $\delta_1^{-1}$ is large compared to $L$ then 
$B(m_0, 4L)$ is homeomorphic to a product of an interval with a $2$-torus, with
${B(m_0, L/10)}$ and
${B(m_0, 2L)}$ homeomorphic to a products of subintervals with the $2$-torus. This
contradicts the assumption that $\gamma$ is contractible in ${B(m_0, 2L)}$ but not in ${B(m_0, L/10)}$.  Hence $(X, x_0)$ is a ray, to which
$(M, m_0)$ is $\delta_1$-close. 
If we take $\delta_1 \ll \delta_3$ then the conclusion of the lemma holds.
\end{proof}

The next lemma is a statement about the lengths of curves $\gamma$ in Lemma \ref{gl3}.

\begin{lemma} \label{gl4}
Given $\vec{K}$ and $L < \infty$, there are $\delta_5 = \delta_5(\vec{K},L) > 0$ and ${\mathcal L} = 
{\mathcal L}(\vec{K},L) > 0$ with the following property.  Suppose that $(M, m_0)$ is a
complete pointed orientable $\vec{K}$-regular Riemannian $3$-manifold so that $(M, m_0)$ has pointed Gromov-Hausdorff distance at most
$\delta_5$ from a ray $(X,x_0)$. Suppose that
there is a loop $\gamma$ in $B(m_0, L/100)$, going through $m_0$, that can be contracted in 
${B(m_0, 2L)}$ but not in 
  ${B(m_0, L/10)}$. Then the length of $\gamma$ is at least ${\mathcal L}$.
\end{lemma}
\begin{proof}
We can assume $K_0 > 0$.
For $\delta_5$ small, we first describe the local geometry of $M$.  In what follows, it will be sufficient to
have estimates with respect to a
metric that is $e^\epsilon$-biLipschitz to the metric $g$ on $M$, for some $\epsilon > 0$.
Fixing $\epsilon$, we can approximate the metric on $M$ by a model
$(\rho, k)$-round metric $\widehat{g}$ in the sense of \cite{Cheeger-Fukaya-Gromov (1992)}.
In doing so we may increase the curvature bound by some factor depending on 
$\epsilon$. Although it's not essential for us, we can assume that the sectional curvatures of $\widehat{g}$ are 
bounded above by $2K_0$ in magnitude \cite[Theorem 2.1]{Rong (1996)}.

In the region of $M$ corresponding to the tip $x_1$ of $X$, the relevant nilpotent group is
$N = \R \times U(1)$ and the relevant covering group is $\Lambda = \Z$, a discrete
subgroup of $N$. From \cite[Theorem 1.3]{Cheeger-Fukaya-Gromov (1992)}, there is an open
subset $V$ of $M$ whose $\Lambda$-cover $\widetilde{V}$ carries a model
$N$-invariant metric. The fixed-point set $\widetilde{c}$ of $U(1)$ will be an infinite geodesic in
$\widetilde{V}$ and we can assume that $\widetilde{V}$ is the
$\rho (2K_0)^{- 1/2}$-neighborhood of $\widetilde{c}$, topologically $\R \times B^2$. 
A point on $\widetilde{c}$ has injectivity radius at least $\rho (2K_0)^{- 1/2}$.
The image $c$ of $\widetilde{c}$, in $V$, is a closed geodesic
that can be considered to lie over $x_1$ in a singular fibration
$V \rightarrow B(x_1, \rho (2K_0)^{- 1/2})$, given by the distance from $c$, with the other fibers being $2$-tori. 
A generator of $\Lambda$ acts by an element
$\left( \tau, e^{i\phi} \right)$ of $N$. For $V$ to be Gromov-Hausdorff close to a ray, 
$\phi$ has to be a small nonzero number
and $\tau$ has to be small relative to $\phi$. Note that $V = \widetilde{V}/\Lambda$ is 
diffeomorphic to a solid torus.  

Given $\beta > 0$, let $\widetilde{W}$ be the $\beta K_0^{- 1/2}$-neighborhood of 
$\widetilde{c}$ in 
$\widetilde{V}$.  We claim that $\beta$ can be chosen small enough, independent of
the other parameters, so that any noncontractible loop on $\partial \widetilde{W}$ (in
the sense that it cannot be contracted in $\partial \widetilde{W}$)
has length at least $\pi \beta K_0^{- 1/2}$.  To see this, if we rescale 
$\widetilde{W}$ so that it is the unit distance neighborhood of $\widetilde{c}$ then
as $\beta \rightarrow 0$, the result approaches the flat isometric product $\R \times
B^2$ uniformly in the $C^{1, \alpha}$-topology. For the latter, the shortest 
noncontractible loop on the boundary of the unit distance neighborhood of
$\R \times \{0\}$
has length $2 \pi$, from which the
claim follows. We will take $\beta  \ll  K_0^{1/2} L/100$. 

Put $W = \widetilde{W}/\Lambda$. If $\sigma$ is a noncontractible loop on $\partial W$ 
that contracts in $W$ then $\sigma$ lifts to a noncontractible loop on $\partial \widetilde{W}$, and hence has length at least $\pi \beta K_0^{- 1/2}$.

There is an open set 
$U \subset M$ with $B(m_0, 3.9L) \subset U \subset B(m_0, 4L)$ so that on $U-W$, the model geometry has a Riemannian submersion 
$F : \left( U - W \right) \rightarrow[a,b)$ whose fibers are flat $2$-tori with diameter
at most $\const \delta_5$ and second fundamental form at most
$\const K_0^{1/2}$ in norm (where $\const$ depends on $\beta$) \cite[Theorem 2.6]{Cheeger-Fukaya-Gromov (1992)}.
Using the (integrable) horizontal distribution to trivialize the fibration as
$G : \left( U - W \right)  \rightarrow [a,b) \times T^2$, we can write
the model metric as $\widehat{g} = dt^2 + h_t$, where 
$t \in [a,b)$ and
$h_t$ is a flat metric on $T^2$.  The shape operator of a fiber is $\frac12 h_t^{-1}
\partial_t h_t$. 

We can assume that
this torus fibration matches up with the torus fibration on a neighborhood of
$\partial W$ in $W$
\cite[Theorem 1.7]{Cheeger-Fukaya-Gromov (1992)}.  Suppose that $\gamma$ is a loop in
$B(m_0, L/100)$, going through $m_0$, that can be
contracted in ${B(m_0, 2L)}$ but not in ${B(m_0, L/10)}$. If the distance from $x_0$ to the tip $x_1$ of the ray is
less than $L/20$ then ${B(m_0, 2L)} - {B(m_0, L/10)}$ lies in the torus-fibered region $U-W$
and $B(m_0, 2L)$ can be retracted to 
$B(m_0, L/10)$, 
which contradicts the
assumption about $\gamma$.  Hence $d(x_0, x_1) \ge L/20$.
For similar reasons, $d(x_0, x_1) \le 5L$.  Hence
$|b-a| \le 10L$.

As $\gamma$ cannot be contracted in ${B(m_0, L/10)}$,
and lies in $B(m_0, L/100) \subset U-W$, it is homotopic to a noncontractible loop in $F^{-1}(m_0)$.
We now sweep $\gamma$ to $\partial W$. That is, writing $G(\gamma(t)) = 
(\gamma_1(t), \gamma_2(t))$, we put
$\widehat{\gamma}(t) = G^{-1}(a, \gamma_2(t))$. The loop $\widehat{\gamma}$ is
noncontractible in $\partial W$. By the bound on the 
shape operators of the fibers of $F$, there is a bound
$l(\widehat{\gamma}) \le e^{\const K_0^{1/2} L} l(\gamma)$. As $\gamma$
contracts in ${B(m_0, 2L)}$, the loop $\widehat{\gamma}$ contracts in $W$.
Hence $l(\gamma) \ge \pi \beta K_0^{- 1/2} e^{- \const K_0^{1/2} L}$, as measured with
$\widehat{g}$. Taking into account that $g$ and $\widehat{g}$ are
$e^\epsilon$-biLipschitz, this proves the lemma.
\end{proof}

We now prove Proposition \ref{3.1}.
Fix a time parameter $u > 0$. 
From Corollary \ref{2.14}, there is a sequence $\{K_i\}_{i=0}^\infty$ so that
for all $s \ge 1$, the metric $g_s(u)$ is $\vec{K}$-regular.

We first describe how we choose parameters. Put $L=1$.
Let $\delta_5$ be the parameter of Lemma \ref{gl4}.
Choose $\delta_3 < \delta_5$ with $\delta_3 \ll 1$.  Let
$\delta_4 = \delta_4(\vec{K},1,\delta_3)$ be the parameter from Lemma \ref{gl3}.
In reference to the parameters $\overline{\delta_1}$ and $\const_1$ of Lemma \ref{gl1}, and ${\mathcal L}$ of Lemma \ref{gl4}, choose $\delta_1 <
\min(\overline{\delta_1}, \delta_4)$ so that $\const_1 \delta_1 < 
\min(\delta_4, {\mathcal L})$. Given this value of $\delta_1$ and in reference to the
parameters $\overline{\delta_2}$ and $\const_2$ of Lemma \ref{gl2}, 
choose $\delta_2 < \min(\overline{\delta_2}, \delta_4)$ so that $\const_2 \delta_2 < 
\min(\delta_4, {\mathcal L})$.

Suppose that the proposition is not true.  Then there is a positive nonincreasing function $\epsilon_s$ with $\lim_{s \rightarrow \infty} 
  \epsilon_s = 0$ so that the
  metric space $(M, {\widehat{d}}_{s,u}, m_0)$ has pointed Gromov-Hausdorff distance
  at most $\epsilon_s$ from a one or two dimensional
complete pointed Alexandrov space with Alexandrov curvature bounded below by 
$-2K_0$.
Take $s_0$ so that $\epsilon_{s_0} \ll \min(\delta_1, \delta_2)$. 
If $(M, {\widehat{d}}_{s_0,u}, m_0)$ is $\delta_1$-close to a one dimensional
complete pointed Alexandrov space then we can apply Lemma \ref{gl1}.  If not,
we can apply Lemma \ref{gl2}. In either case we get a loop $\gamma$ through
$m_0$, with length at most $\min(\delta_4, {\mathcal L})$, that is not
contractible in ${B_{\widehat{d}_{s_0,u}}(m_0, 1)}$.

As $M$ is diffeomorphic to $\R^3$, there is some $\Delta > 1$
so that $\gamma$ can be contracted in ${B_{\widehat{d}_{s_0,u}}(m_0, \Delta)}$.
Let ${\mathcal C}$ be the set of $s \ge s_0$ so that $\gamma$ can be contracted in ${B_{\widehat{d}_{s,u}}(m_0, 1)}$. By the right-hand inclusion of
(\ref{2.20}), large values of $s$ are in ${\mathcal C}$, and if
$s \in {\mathcal C}$ then $s^\prime \in {\mathcal C}$ for all $s^\prime > s$.
By (\ref{2.20}), ${\mathcal C}$ is open in $\R$. Hence it is a
half-open interval $(s_1, \infty)$ for some $s_1 \ge s_0$.  

In particular, $\gamma$ cannot be contracted in ${B_{\widehat{d}_{s_1,u}}(m_0, 1/10)}$.
On the other hand, if $s$ is slightly greater than $s_1$ then $\gamma$ can be contracted in
${B_{\widehat{d}_{s,u}}(m_0, 1)}$. Then by the left-hand inclusion of (\ref{2.20}),
$\gamma$ can be contracted in ${B_{\widehat{d}_{s_1,u}}(m_0, 2)}$. By (\ref{2.19}), the length of $\gamma$
with respect to $g_{s_1}(u)$ is still less than
$\min(\delta_4, {\mathcal L})$. 

Lemma \ref{gl3} implies that
$(M, m_0)$ has pointed Gromov-Hausdorff distance at most $\delta_3 < \delta_5$ from some ray
$(X,x_0)$. Lemma \ref{gl4} now implies that
the length of $\gamma$, with respect to $g_{s_1}(u)$, is at least
${\mathcal L}$.  This is a contradiction.
  
\section{Cubic volume growth} \label{sect4}

\begin{proposition} \label{4.1}
  Under the hypotheses of Proposition \ref{1.6}, and with reference to
  Proposition \ref{3.1},
both
  $(M, g_0)$ and $(M_\infty, g_\infty(u))$ have cubic volume growth.
  In addition, each tangent cone at infinity of $(M_\infty, g_\infty(u))$
  is isometric to the tangent cone at infinity
  $T_\infty M =
  \lim_{i \rightarrow \infty} \left( M, m_0,
  s_i^{- \: \frac12} d_0 \right)$
  of $M$.
\end{proposition}
\begin{proof}
We know that
the pointed limit
  $\lim_{i \rightarrow \infty} (M, g_{s_i}(\cdot), m_0)$ exists as a
  Ricci flow $(M_\infty, g_\infty(\cdot), m_\infty)$ on a
  pointed $3$-manifold $(M_\infty, m_\infty)$. We claim first that
  $(M, g_0)$ has cubic volume growth. Fix $u > 0$. Given $R > 0$, put
  $U_i = B_{\widehat{d}_{s_i,u}}(m_0, R)$ and
  $C_R = \vol(B(m_\infty, R), g_\infty(u))$. Then 
  for large $i$, using (\ref{2.17}) we have
  \begin{equation} \label{4.2}
    s_i^{- \: \frac32} \vol(U_i, d_0) = \vol(U_i, s_i^{- \: \frac12} d_0) \ge
    \vol(U_i, \widehat{d}_{s_i,u}) \ge \frac12 C_R,
  \end{equation}
  where $\vol$ denotes the $3$-dimensional Hausdorff mass computed with the
  given metric.
  Also from (\ref{2.18}), we have
  $U_i \subset B_{d_0}(m_0, s_i^{\frac12} (R+C^\prime \sqrt{u}))$. Hence
  \begin{equation} \label{4.3}
    \vol(B_{d_0}(m_0, s_i^{\frac12} (R+C^\prime \sqrt{u}))) \ge
    \frac12 C_R s_i^{\frac32}.
  \end{equation}
  Since $r^{-3} \vol(B(m_0, r), g_0)$ is nonincreasing in $r$, it follows that
  there is some $v_0 > 0$ so that for all $r > 0$, we have
  $\vol(B(m_0, r), d_0) \ge v_0 r^3$.

Let $d_\infty$ denote the metric on $T_\infty M$. 
Let $\widehat{d}_{\infty, u}$ denote the metric on $(M_\infty, g_\infty(u))$.
From (\ref{2.17}), we have
\begin{equation} \label{4.4}
  d_\infty - C^\prime \sqrt{u} \le \widehat{d}_{\infty,u} \le
  d_\infty
\end{equation}
on $T_\infty M - B_{d_\infty}(\star_\infty, C^\prime \sqrt{u})$. Hence
the tangent cone at infinity of $(M_\infty, g_\infty(u))$ is unique and
  is isometric to $(T_\infty M, d_\infty)$. 
  \end{proof}

\begin{proposition} \label{4.5}
  If $\{s_i\}_{i=1}^\infty$
  is any sequence tending to infinity then after passing to a subsequence,
  there is a pointed limit
  $\lim_{i \rightarrow \infty} (M, g_{s_i}(\cdot), m_0)$ as a
  Ricci flow on a pointed $3$-manifold, defined for times $u \in (0, \infty)$. 
  \end{proposition}
\begin{proof}
Put $v_\infty = \lim_{r \rightarrow \infty} r^{-3} \vol(B_{d_0}(m_0, r),
g_0) > 0$, the asymptotic volume ratio of $(M, g_0)$.
Fix $u > 0$. For any $s > 1$, from (\ref{2.17}) a tangent cone at infinity of
$(M, g_s(u))$ is isometric to a tangent cone at infinity of
$(M, g_0)$. Hence the asymptotic volume ration of $(M, g_s(u))$ is $v_\infty$.
Given $R > 0$, the Bishop-Gromov inequality implies that
  $\vol(B_{\widehat{d}_{s,u}}(m_\infty, R), \widehat{d}_{s,u}) \ge
v_\infty R^3$. As $|\Rm(g_s(u))| \le \frac{C}{u}$, the claim follows
from the Hamilton compactness theorem.
\end{proof}

The next lemma will be used in Section \ref{sect5}.

\begin{lemma} \label{4.6}
  A three dimensional complete gradient expanding soliton $(M,g)$
  with bounded sectional curvature, $c$-pinched nonnegative Ricci
  curvature, and cubic volume growth, must be isometric to
  flat $\R^3$.
\end{lemma}
\begin{proof}
  If $(M,g)$ is flat then because of the cubic volume growth, it must be
  isometric to $\R^3$. Hence we can assume that $\Ric(M,g) > 0$.
  From \cite[Proposition 3.1]{Ni (2005)}, $(M,g)$ has exponential curvature
  decay. Fix a basepoint $m_0$.
  We can find a sequence $\alpha_i \rightarrow \infty$ so that
  $\{(M, \alpha_i^{-2} g, m_0)\}_{i=1}^\infty$ converges in the pointed
  Gromov-Hausdorff topology
  to a three dimensional tangent cone at infinity $(X_\infty, x_\infty)$ of $(M,g)$.
  In particular, $(X_\infty, x_\infty)$ is a cone over a connected compact surface.
  Because of the quadratic curvature decay, after passing to a further
  subsequence we can assume that there is a
  $W^{2,p}$-regular Riemannian metric on $X_\infty - x_\infty$, along with
  convergence of metrics in 
  the weak $W_{loc}^{2,p}$-topology;  c.f. \cite{Kasue (1989)} and
  \cite[Sections 4 and 5]{Petersen (1997)}.
  From the weak $W_{loc}^{2,p}$-convergence and the exponential curvature
  decay of $(M,g)$,
  the Riemannian metric on $X_\infty - x_\infty$ is flat.
  Hence
  $X_\infty$ is a cone over the round $S^2$ or its $\Z_2$-quotient
  $\R P^2$.  As $M$ was orientable, the second possibility cannot occur,
  so $X_\infty$ is the flat $\R^3$. Then by
  \cite[Theorem 0.3]{Colding (1997)}, $(M,g)$ is flat,
  which is a contradiction.
\end{proof}

\begin{remark}
Under the additional
assumption of 
nonnegative sectional curvature, Lemma \ref{4.6} was proven in
\cite{Chen-Zhu (2000)}.
\end{remark}

\section{Proof of Theorem \ref{1.3}} \label{sect5}

\begin{proposition} \label{5.1}
If $(M, g_0)$ has nonnegative sectional curvature then Conjecture \ref{1.1} holds.
    \end{proposition}
\begin{proof}
  It is enough to prove that Conjecture \ref{1.2} holds, so
  we will assume that $\Ric_M > 0$, with $M$ noncompact,
  and derive a contradiction.
  Using Proposition \ref{4.1} and \cite[Theorem 1.2]{Schulze-Simon (2013)},
  there is a blowdown limit $(M_\infty, g_\infty(\cdot), m_\infty)$
  that is an gradient expanding soliton.
  From Lemma \ref{4.6}, it must be isometric
 to $\R^3$.
  Hence $T_\infty M$ is isometric to $\R^3$.
  By \cite[Theorem 0.3]{Colding (1997)},
   $(M, g_0)$ is isometric to $\R^3$, which contradicts our assumption that
  $\Ric_M > 0$. 
  \end{proof}

\begin{remark} \label{5.2}
To clarify a technical point, 
in \cite{Chen-Zhu (2000)} use is made of
\cite[Theorem 16.5]{Hamilton (1995)} to say that
$A = \limsup_{t \rightarrow \infty} t \|\Rm(g(t))\|_\infty$ is positive.
The proof of 
\cite[Theorem 16.5]{Hamilton (1995)} is based on
\cite[Theorem 16.4]{Hamilton (1995)}, which has a similar conclusion
without an assumption of positivity of curvature, but
whose proof is only valid in the
compact case (since it invokes the diameter). In fact, there are
noncompact counterexamples to \cite[Theorem 16.4]{Hamilton (1995)}.
With nonnegative curvature operator, the trace Harnack inequality
directly implies that $A > 0$ for nonflat solutions.
    \end{remark}
    
    \begin{lemma} \label{fixed}
      If $(M, g)$ has $c$-pinched nonnegative Ricci curvature then for any tangent vector $v$,
      we have 
      \begin{equation} \label{RR}
      R g(v,v) \le \left( 1 + \frac{2}{c} \right) \Ric(v,v).
      \end{equation}
    \end{lemma}
    \begin{proof}
      Working in a tangent space $T_mM$, we can restrict to unit vectors $v$.  Lagrange
      multipliers show that the maximum of $R g(v,v) - \left( 1 + \frac{2}{c} \right) \Ric(v,v)$, over unit vectors in $T_mM$, is realized at a unit eigenvector of the
      Ricci operator. Let $r_1 \le r_2 \le r_3$ be the eigenvalues of the Ricci operator.
      As $r_2 \le r_3 \le \frac{1}{c} r_1$, we have 
      \begin{equation}
      R = r_1 + r_2 + r_3 \le 
      \left( 1 + \frac{2}{c} \right) r_1 \le \left( 1 + \frac{2}{c} \right) r_i
        \end{equation}
        for each $i \in \{1,2,3\}$.  This proves the lemma.
      \end{proof}

\begin{proposition} \label{5.3}
 If $(M, g)$ has quadratic curvature decay then Conjecture \ref{1.1} holds.
    \end{proposition}
\begin{proof}
  We will assume that $\Ric_M > 0$, with $M$ noncompact,
  and derive a contradiction. 
  Given $\alpha > 1$,  consider the rescaled metric $\alpha^{-2} g$.
  Using Proposition \ref{4.1}, there is a sequence
  $\{\alpha_i\}_{i=1}^\infty$ tending to infinity so that
  $\{\left( M, \alpha_i^{-2} g, m_0 \right)
  \}_{i=1}^\infty$
  has a limit $T_\infty M = \lim_{i \rightarrow \infty}
  \left( M, \alpha_i^{-2} g_0, m_0 \right)$
  in the pointed Gromov-Hausdorff topology,
  where $T_\infty M$ is a three dimensional 
  metric cone with a connected link
  \cite[Theorem 7.6]{Cheeger-Colding (1996)}.
  Furthermore, from the quadratic curvature decay, $T_\infty M$ is a
  smooth manifold away from the vertex, where it has a 
  $W^{2,p}_{loc}$, $p < \infty$, or $C^{1, \alpha}_{loc}$, $\alpha \in (0,1)$, metric
  $g_\infty$, and the
  convergence to $g_\infty$ is in the weak $W^{2,p}_{loc}$-topology.
  The inequality (\ref{RR}) will pass to such a limit. 
  Since $T_\infty M$ is a cone, if $\partial_r$
  denotes the radial vector field then from the cone structure,
  $\Ric_{g_\infty}(\partial_r, \partial_r) = 0$.
  Then by (\ref{RR}), we conclude that $R_{g_\infty} = 0$.
  This means that $\Ric_{g_\infty}$ vanishes, so $g_\infty$ is smooth and flat.
  Hence $T_\infty M$ is a cone over the round $S^2$ or
  $\R P^2$. Since $M$ is orientable, $T_\infty M$ must be a cone over the
  round $S^2$, and hence is isometric to
  $\R^3$. 
  By \cite[Theorem 0.3]{Colding (1997)},
    $(M, g)$ is isometric to $\R^3$, which contradicts our assumption that
  $\Ric_M > 0$. 
\end{proof}

\section{Proof of Theorem \ref{1.4}} \label{sect6}

To prove Theorem \ref{1.4} we will use a rescaling argument as in the
proof of Proposition \ref{5.3}.  The rescalings no longer have
uniform local double sided bounds on their
curvatures, so we need a different
convergence result.  This will come from \cite{Lebedeva-Petrunin}, which
provides a weak convergence of curvature operators.  It turns out
that this is enough to obtain a contradiction.

We recall some results from \cite{Lebedeva-Petrunin}.
Given an $n$-dimensional Riemannian manifold $(M, g)$,
let $\Riem$ be the curvature
operator of $M$ and let $\star_M : \Lambda^{n-2}(TM) \rightarrow
\Lambda^2(TM)$ be Hodge duality. 
Given $C^1$-functions
$\{f_j\}_{j=1}^{n-2}$ on $M$, put
\begin{equation} \label{6.1}
  \sigma = \star_M (\nabla f_1 \wedge \nabla f_2 \wedge \ldots \wedge \nabla f_{n-2})
\end{equation}
and define
\begin{equation} \label{6.2}
  r_M(f_1, \ldots, f_{n-2}) = \langle \sigma, \Riem(\sigma) \rangle \: \dvol_M,
\end{equation}
a measure on $M$.

Suppose that
$\{(M_i, g_i)\}_{i=1}^\infty$ is a sequence of 
$n$-dimensional pointed complete
Riemannian manifolds with sectional curvatures uniformly bounded below, that
converges to an $n$-dimensional pointed Alexandrov space
$X_\infty$ in the Gromov-Hausdorff topology.
There is a notion of a test function $f_\infty$ on $X_\infty$.
Given $C^1$-functions $\{f_i\}_{i=1}^\infty$ on the $M_i$'s, there is a notion of
the sequence $C^1_\delta$-converging to $f_\infty$.

The main result of \cite{Lebedeva-Petrunin} is the following.  Suppose
that for each $i$,  $\{f_{i,j}\}_{1 \le j \le n-2}$ is a collection of
$C^1$-functions on $M_i$. Suppose that for each $j$, there is a
$C^1_\delta$-limit
$\lim_{i \rightarrow \infty} f_{i,j} = f_{\infty, j}$, where
  $f_{\infty, j}$ is a test function on $X_\infty$. Then there is a weak limit
\begin{equation} \label{6.3}
  \lim_{i \rightarrow \infty} r_{M_i}( f_{i,1}, \ldots, f_{i,n-2}) =
  r_{X_\infty}( f_{\infty,1}, \ldots, f_{\infty,n-2}).
\end{equation}
Furthermore, the measure $r_{X_\infty}( f_{\infty,1}, \ldots, f_{\infty,n-2})$
is intrinsic to $X_\infty$.  It vanishes on the strata of
$X_\infty$ with codimension greater than two,
and has descriptions on the codimension-two stratum and the set of regular points. Similarly, there is a measure $R_{X_\infty}$ on $X_\infty$ to which the
scalar curvature measures converge, i.e.
$\lim_{i \rightarrow \infty} R_{M_i} \dvol_{M_i} = R_{X_\infty}$ in the weak
topology.

The preceding constructions can also be carried out locally.

\begin{lemma}
Suppose that $M$ is a $3$-manifold with $c$-pinched nonnegative Ricci curvature. 
Given $f \in C^1(M)$, put $V = \nabla f$. Then
      \begin{equation} \label{RR2}
      R g(V,V) \dvol_M \le \left( 1 + \frac{2}{c} \right) \left( \frac12 R \dvol_M
      g(V,V) -
      r_M(f) \right).
      \end{equation}
\end{lemma}
\begin{proof}
The proof is similar to that of Lemma \ref{fixed}.  Fixing $m \in M$ and
restricting to unit vectors $V_m \in T_m M$, using Lagrange multipliers one sees that it
is enough to check the inequality when $V_m$ is an eigenvector associated to the quadratic form
coming from the difference of the two sides of (\ref{RR2}). One finds that these
eigenvectors are the eigenvectors of the Ricci operator, in which case
(\ref{RR2}) reduces to (\ref{RR}).
\end{proof}

\begin{proposition} \label{6.4} If there is some $A < \infty$ so that
      the sectional curvatures of $(M, g)$ satisfy
      $K(m) \ge - \: \frac{A}{d(m,m_0)^2}$, where $m_0$ is
      some basepoint, then Conjecture \ref{1.1} holds.
\end{proposition}
\begin{proof}
  We will assume that $\Ric_M > 0$, with $M$ noncompact,
  and derive a contradiction. From Proposition \ref{4.1}, there is a sequence
  $\{\alpha_i\}_{i=1}^\infty$ tending to infinity so that putting
  $g_i = \alpha_i^{-2} g$ and 
  $M_i = \left( M, g_i \right)$, the sequence 
  $\{ \left( M_i, m_0 \right)
  \}_{i=1}^\infty$
  converges to a three-dimensional metric cone
  $(X_\infty, x_\infty)$ in the pointed Gromov-Hausdorff topology.
  From the curvature assumption, 
  the cone $X_\infty$ has curvature bounded
  below by the function $- \: \frac{A}{d(x, x_\infty)^2}$
  in the Alexandrov sense.
  As a locally Alexandrov space, the cone will have no boundary points, i.e.
  no codimension-one stratum.
  Let $\Sigma_\infty$ denote the link of the cone, so that
  $X_\infty = \cone(\Sigma_\infty)$. 
  Then $\Sigma_\infty$ is a connected
  Alexandrov surface with curvature bounded below by $-A$. The underlying topological space of $\Sigma_\infty$ is a $2$-manifold $Y$ without boundary, which hence admits a smooth structure.
  Let $\omega_Y$ denote the curvature measure on $Y$, in the sense of
  \cite{Reshetnyak (1993)}.  (If $Y$ is a smooth Riemannian $2$-manifold then
  $\omega_Y = K \dvol_Y$, where $K$ is the Gaussian curvature.)
  
\begin{lemma} \label{6.5}
  Let $\partial_r$ denote the radial vector field on $X_\infty$.
  Then
\begin{equation} \label{6.6}
  r_{X_\infty}(f) = (\partial_r f)^2 dr \wedge (d\omega_{Y} -
  \dvol_{Y}),
  \end{equation}
where $d\omega_{Y}$ is the curvature measure of the Alexandrov surface $Y$
and $\dvol_Y$ is the two-dimensional Hausdorff measure of $Y$. Also,
\begin{equation} \label{6.7}
R_{X_\infty} = 2 dr \wedge (d\omega_{Y} -
\dvol_{Y}).
\end{equation}
\end{lemma}
\begin{proof}
From
\cite[Section 1 and Appendix A]{Richard (2018)}, there is a
$1$-parameter family of smooth Riemannian metrics $\{h_s\}_{s \in (0, \epsilon)}$ on $Y$ so that $\lim_{s \rightarrow 0} (Y, h_s) = \Sigma_\infty$ in the Gromov-Hausdorff topology, and the curvature of $(Y, h_s)$ is bounded below by $-A$.
For $s \in (0, \epsilon)$,
let $Y_s$ denote $Y$ with the Riemannian metric $h_s$.
We first compute $r_{\cone(Y_s)}$.
Writing $\cone(Y_s) - \star = (0, \infty) \times Y_s$, 
if $V$ is a vector field on $Y_s$ then we can also consider it to be a
vector field on $\cone(Y_s) - \star$. 
We have $\Riem(\partial_r \wedge V) = 0$.
If $V$ and $W$ are vector fields on $Y_s$ then
$\Riem(V \wedge W) = \frac{K_s-1}{r^2} V \wedge W$, where
$K_s$ is the Gaussian curvature of $Y_s$.
Hence if $f$ is the radial function on $\cone(Y_s)$ then
\begin{align} \label{6.8}
  r_{\cone(Y_s)} (f) = &
  \langle \star_{\cone(Y_s)} \partial_r, \Riem(\star_{\cone(Y_s)} \partial_r)
  \rangle \:
  \dvol_{\cone(Y_s)} \\
  = &
  \frac{K_s-1}{r^2} \langle \star_{\cone(Y_s)} \partial_r, \star_{\cone(Y_s)}
  \partial_r \rangle \: r^2 dr \wedge \dvol_{Y_s} \notag \\
  = & 
  (K_s-1) dr \wedge \dvol_{Y_s} = dr \wedge (K_s \dvol_{Y_s} - \dvol_{Y_s}). \notag
\end{align}
Then in general,
\begin{equation} \label{6.9}
  r_{\cone(Y_s)} (f) = (\partial_r f)^2 dr \wedge (K_s \dvol_{Y_s} - \dvol_{Y_s}).
\end{equation}

As $s \rightarrow 0$, we have pointed Gromov-Hausdorff convergence
$\lim_{s \rightarrow 0} \cone(Y_s) = X_\infty$. Working locally on
$X_\infty$, say on an annular region $a \le r \le A$, there is a weak
limit $\lim_{s \rightarrow \infty} r_{\cone(Y_s)} = r_{X_\infty}$, which
gives (\ref{6.6}).

As
\begin{equation} \label{6.10}
  R_{\cone(Y_s)} \dvol_{\cone(Y_s)} = \frac{2(K_s-1)}{r^2} r^2 dr \wedge \dvol_{Y_s} =  2(K_s-1) dr \wedge \dvol_{Y_s},
\end{equation}
equation (\ref{6.7}) follows.
\end{proof}

If $f$ is the radial function $r$ on
$X_\infty$ then from (\ref{6.6}) and (\ref{6.7}), we have
\begin{equation} \label{6.11}
  R_{X_\infty} = 2 r_{X_\infty}(f).
  \end{equation}
  Returning to the sequence $\{M_i\}_{i=1}^\infty$,
  the inequality (\ref{RR2}) will pass to the weak limit, so (\ref{6.11}) implies that
  $R_{X_\infty} = 0$. Equation (\ref{6.7}) gives
\begin{equation} \label{6.20}
  d\omega_Y = \dvol_Y.
\end{equation}

Integrating (\ref{6.20}) over $Y$ shows that the Euler characteristic of $Y$ is positive.  By Perelman stability, $\cone(X_\infty) - \star$ is orientable, so
$Y$ is a $2$-sphere.  As an Alexandrov surface, the Alexandrov geometry on $Y$
comes from a Riemannian metric of the form $e^{2\phi} g_{S^2}$ which is
subharmonic in the sense of \cite[Section 7]{Reshetnyak (1993)}. Equation (\ref{6.20}) becomes
\begin{equation} \label{6.21}
  \triangle_{S^2} \phi - 1 = - e^{2\phi},
  \end{equation}
where $\triangle_{S^2} \phi$ is {\it a priori} a measure on $S^2$ and 
$e^{2\phi}$ is an $L^1$-function.  From (\ref{6.20}) the curvature measure $d\omega_Y$
is absolutely continuous with respect to the Riemannian density on the round $S^2$, so
\cite[Proposition 2.8]{Ambrosio-Bertrand (2016)} implies that $e^{2\phi}$ is
$L^p$-regular on $S^2$ for all $p < \infty$.  Then (\ref{6.21}) implies that
$\phi$ is $W^{2,p}$-regular, hence $C^{1, \alpha}$-regular for all $\alpha \in (0,1)$.
We can now bootstrap (\ref{6.21}) to conclude that $\phi$ is smooth on $S^2$.  Then
(\ref{6.21}) becomes the statement that $e^{2\phi} g_{S^2}$ has constant curvature $1$.
Thus $Y$ is isometric to
the round $S^2$. Hence $X_\infty$ is the flat $\R^3$. By
\cite[Theorem 0.3]{Colding (1997)}, $(M, g)$ is flat, which is a
contradiction.
\end{proof}


\begin{thebibliography}{10}
  
  \bibitem{Ambrosio-Bertrand (2016)} L. Ambrosio and J. Bertrand,
  ``On the regularity of Alexandrov surfaces with curvature bounded below'',
  Anal. Geom. Metr. Spaces 2016, p. 282-287 (2016)
  
  \bibitem{Brendle-Huisken-Sinestrari (2011)}
    S. Brendle, G. Huisken and C. Sinestrari,
    ``Ancient solutions to the Ricci flow with pinched curvature'',
    Duke Math. J. 158, p. 537-551 (2011)
    
    \bibitem{Cheeger-Colding (1996)} J. Cheeger and T. Colding, ``Lower bounds on Ricci
    curvature and the almost rigidity of warped products'', Ann. Math. 144, p. 189-237 (1996)

  \bibitem{Cheeger-Fukaya-Gromov (1992)}
    J. Cheeger, K. Fukaya and M. Gromov,
    ``Nilpotent structures and invariant metrics on collapsed manifolds'',
    J. Amer. Math. Soc. 5, p. 327-372 (1992)
    (1992)

  \bibitem{Cheeger-Gromov (1985)} J. Cheeger and M. Gromov,
    ``On the characteristic numbers of complete manifolds of bounded curvature
    and finite volume'', in \underline{Differential Geometry and Complex
      Analysis}, Springer, Berlin,
    pp. 115-154 (1985)
    
  \bibitem{Cheeger-Gromov (1990)} J. Cheeger and M. Gromov,
    ``Collapsing Riemannian manifolds while keeping their curvature bounded II''
    J. Diff. Geom. 32, p. 269-298 (1990)
  
  \bibitem{Cheeger-Tian (2005)} J. Cheeger and G. Tian,
    ``Curvature and injectivity radius estimates for Einstein $4$-manifolds'',
    J. Amer. Math. Soc. 19, p. 487-525 (2005)

  \bibitem{Chen-Zhu (2000)} B.-L. Chen and X.-P. Zhu,
    ``Complete Riemannian manifolds with pointwise pinched curvature'',
    Inv. Math. 140, p. 423-452
    (2000)

  \bibitem{Chow-Lu-Ni (2006)} B. Chow, P. Lu and L. Ni,
\underline{Hamilton's Ricci flow}, Amer. Math. Soc., Providence (2006)
  
\bibitem{Colding (1997)} T. Colding,
``Ricci curvature and volume convergence'', Ann. Math 145, p. 477-501
  (1997)

\bibitem{Deruelle (2016)} A. Deruelle,
  ``Smoothing out positively curved metric cones by Ricci expanders'',
  Geom. and Funct. Anal. 26, p. 188-249 (2016)

\bibitem{Deruelle-Schulze-Simon (2022)} A. Deruelle, F. Schulze and M. Simon,
  ``Initial stability estimates for Ricci flow and three dimensional
  Ricci-pinched manifolds'', preprint,
  https://arxiv.org/abs/2203.15313 (2022)
  
\bibitem{Fukaya (1990)} K. Fukaya,
  ``Hausdorff convergence of Riemannian manifolds and its applications'',
  in \underline{Recent Topics in Differential and Analytic Geometry},
  ed. T. Ochiai, Math. Soc. of Japan, Tokyo, p. 143-238
  (1990)

\bibitem{Grove-Karcher (1973)} K. Grove and H. Karcher,
``How to conjugate C1-close group actions'',
Math. Z. 132, p. 11-20 (1973)
  
\bibitem{Hamilton (1982)} R. Hamilton,
``Three-manifolds with positive Ricci curvature'', J. Diff. Geom. 17, p. 255-306
  (1982)

\bibitem{Hamilton (1994)} R. Hamilton,
  ``Convex hypersurfaces with pinched second fundamental form'',
  Comm. Anal. Geom. 2, p. 167-172 (1994)
  
\bibitem{Hamilton (1995)} R. Hamilton,
  ``The formation of singularities in the Ricci flow'', in
  Surveys in Differential Geometry 2, International Press, Boston, p. 7-136
  (1995)
  
\bibitem{Hilaire (2014)} C. Hilaire,
  ``Ricci flow on Riemannian groupoids'', preprint,
  https://arxiv.org/abs/1411.6058
  (2014)

  \bibitem{Kasue (1989)} A. Kasue, ``A convergence theorem for Riemannian manifolds and some applications'',
Nagoya Math. J. 114, p. 21-51 (1989)

\bibitem{Kleiner-Lott (2008)} B. Kleiner and J. Lott,
``Notes on Perelman's papers'', Geom. Top. 12, p. 2587-2855 
  (2008)

\bibitem{Kleiner-Lott (2014)} B. Kleiner and J. Lott,
``Locally collapsed $3$-manifolds'', Ast\'erisque 365, p. 7-99
  (2014)
  
 \bibitem{Kleiner-Lott (2014b)} B. Kleiner and J. Lott,
 ``Geometrization of three-dimensional orbifolds via Ricci flow'',
 Asterisque 365, p. 101-177 (2014)

\bibitem{Lebedeva-Petrunin} N. Lebedeva and A. Petrunin,
  ``Curvature tensor of smoothable Alexandrov spaces'',
preprint, https://arxiv.org/abs/2202.13420 (2022)

\bibitem{Lee-Topping (2022)} M.-C. Lee and P. Topping,
  ``Three-manifolds with non-negatively pinched Ricci curvature'',
  preprint, https://arxiv.org/abs/2204.00504 (2022)

\bibitem{Lott (2007)} J. Lott,
      ``On the long-time behavior of type-III Ricci flow solutions'',
      Math. Annalen 339, p. 627-666 
      (2007)

    \bibitem{Lott-Zhang (2016)} J. Lott and Z. Zhang,
      ``Ricci flow on quasiprojective manifolds II'', J. Eur. Math. Soc. 18,
      p. 1813-1854 
      (2016)

    \bibitem{Lu (2010)} P. Lu,
      ``A local curvature bound in Ricci flow'', Geom. Top. 14, p. 1095-1110
      (2010)

    \bibitem{Ni (2005)} L. Ni,
``Ancient solutions to K\"ahler-Ricci flow'', Math. Res. Lett. 12, p. 633–654
      (2005)
      
      \bibitem{Ni-Wu (2007)} L. Ni and B. Wu,
      ``Complete manifolds with nonnegative curvature operator'',
Proc. Amer. Math. Soc. 135 (2007), p. 3021-3028 (2007) 

\bibitem{Petersen (1997)} P. Petersen,
``Convergence theorems in Riemannian geometry'', in \underline{Comparison geometry},
MSRI Publ. 30, Cambridge Univ. Press, Cambridge, p. 167-202 (1997)

    \bibitem{Reshetnyak (1993)} Y. Reshetnyak,
      ``Two-dimensional manifolds of bounded curvature'', in
      \underline{Geometry IV}, Springer-Verlag, New York
      (1993)
      
      \bibitem{Rong (1996)} X. Rong,
      ``On the fundamental groups of manifolds of positive sectional curvature'', Ann. of Math., p. 397-411 (1996)

    \bibitem{Richard (2018)} T. Richard,
      ``Canonical smoothing of compact Alexandrov surfaces'',
      Ann. Sci. ENS 51, p. 263-279
      (2018)
  
    \bibitem{Schoen-Yau (1982)} R. Schoen and S.-T. Yau,
      ``Complete three dimensional manifolds with positive Ricci curvature and
      scalar curvature'', in \underline{Seminar on Differential Geometry},
      ed. S.-T. Yau, Annals of Math. Studies 102, Princeton University Press,
      Princeton
      (1982) 

    \bibitem{Schulze-Simon (2013)} F. Schulze and M. Simon,
      ``Expanding solitons with non-negative curvature operator coming out of
      cones'', Math. Z. 275, p. 625–639 (2013)
      
    \bibitem{Scott (1983)} P. Scott, ``The geometries of $3$-manifolds'',
      Bull. London Math. Soc. 15, p. 401-487 (1983)
  
    \bibitem{Vogt (1989)} E. Vogt, ``A foliation of $\R^3$ and other
      punctured $3$-manifolds by circles'',
Publ. Math. IHES 69, p. 215-232
      (1989)
\end{thebibliography}
\end{document}